\documentclass[reqno]{amsart}
\usepackage{amssymb}
\usepackage{ifpdf}
\ifpdf
 \usepackage[hyperindex,pagebackref]{hyperref}
\else
 \expandafter\ifx\csname dvipdfm\endcsname\relax
 \usepackage[hypertex,hyperindex,pagebackref]{hyperref}
 \else
 \usepackage[dvipdfm,hyperindex,pagebackref]{hyperref}
 \fi
\fi
\allowdisplaybreaks[4]
\numberwithin{equation}{section}
\theoremstyle{plain}
\newtheorem{theorem}{Theorem}[section]
\newtheorem{open}{Open Problem}[section]

\newtheorem{definition}{Definition}[section]
\newtheorem{lemma}{Lemma}[section]
\newtheorem{corollary}{Corollary}[section]
\theoremstyle{remark}
\newtheorem{remark}{Remark}[section]

\begin{document}

\title[$k$-gamma and $k-$ digamma functions]
{On some properties of special functions involving $k$-gamma and $k-$digamma functions}

\author[L. Yin]{Li Yin}
\address[L. Yin]{College of Science, Shandong University of Aeronautics(Binzhou University), Binzhou City, Shandong Province, 256603, China}
\email{\href{mailto: L. Yin<yinli_79@163.com>}{yinli\_79@163.com}}
\email{\href{mailto: L. Yin<yinli1979@sdua.edu.cn>}{yinli1979@sdua.edu.cn}}

\author[J.-M. Zhang]{Jumei Zhang}
\address[J.-M. Zhang]{College of Science, Shandong University of Aeronautics(Binzhou University), Binzhou City, Shandong Province, 256603, China}
\email{\href{mailto:
J. M. Zhang<wdzjm02@163.com>}{wdzjm02@163.com}}



\begin{abstract}
Based on $k$-gamma and $k-$digamma functions, we show four series expansions to the Furdui-type integral related to
Riemann zeta function and hypergeometric function, and also present
 some new identities, series expansions and inequalities on the Hadamard $k$-gamma function and the Nielsen $k$-beta function. Finally, we also pose an open problem.
\end{abstract}

\thanks{Corresponding author: Li Yin}

\maketitle

\section{Introduction}
The Euler gamma function is defined all positive real numbers $x$ by
\begin{equation*}\Gamma(x)=\int_0^\infty t^{x-1}e^{-t}dt.\end{equation*}
It is common knowledge that the logarithmic derivative of $\Gamma(x)$ is called the psi or digamma function, and $\psi^{(m)}(x)$ for $m\in\mathbb{N}$ are known as the polygamma functions. That is
\begin{equation*}
\psi(x)=\frac{d}{dx}\ln \Gamma(x)=\frac{\Gamma'(x)}{\Gamma(x)}=-\gamma-\frac{1}{x}+\sum_{n=1}^{\infty}\frac{x}{n(n+x)},
\end{equation*}
where $\gamma=0.5772\ldots$ is the Euler-Mascheroni constant. The polygamma functions $\psi^{(m)}(x)$ for $m\in\mathbb{N}$ are defined by
$$
\psi ^{(m)} (x) = \frac{{d^m }}{{dx^m }}\psi (x) = ( - 1)^m m!\sum\limits_{n = 0}^\infty  {\frac{1}{{(n + x)^{m + 1} }}} ,x > 0.
$$The gamma, digamma and polygamma functions play an important role in the theory of special functions, and have many applications in other many branches, such as statistics, fractional differential equations, mathematical physics and theory of infinite series. The reader may see references \cite{cyin,dyin,dp}. some of the work about the complete monotonicity, convexity and concavity, and inequalities of these special functions may refer to \cite{as,a2,ba1,ba2,gq,gq2,gqs,gzq,m,qc,qg,qgg,qg2,qg3,yin,yinhuang,yinhuangsong} and other relevant references..

In 2007, Diaz and Pariguan \cite{dp} defined the $k-$analogue of the gamma function for $k>0$ and $x>0$ as
\begin{equation*}
\Gamma_k(x)=\int_0^\infty t^{x-1}e^{-\frac{t^k}{k}}dt=\lim_{n\rightarrow \infty}\frac{n!k^n(nk)^{\frac{x}{k}-1}}{x(x+k)\cdots(x+(n-1)k)},
\end{equation*}
where $\lim_{k\rightarrow 1}\Gamma_k(x)=\Gamma(x)$. Similarly, we may define the $k-$analogue of the digamma and polygamma functions as
$$\psi_k(x)=\frac{d}{dx}\ln \Gamma_k(x) \quad\mathrm{and} \quad\psi_k^{(m)}(x)=\frac{d^m}{dx^m}\psi_k(x).$$

It is well known that the $k-$analogues of the digamma and polygamma functions satisfy the following recursive formula and series identities (See \cite{dp})
\begin{equation}\label{1.1-eq}
\Gamma_k(x+k)=x\Gamma_k(x), \quad x>0,
\end{equation}

\begin{align}\label{1.2-eq}
\psi_k(x)&=\frac{\ln k-\gamma}{k}-\frac{1}{x}+\sum_{n=1}^{\infty}\frac{x}{nk(nk+x)}\\
&=\frac{\ln k-\gamma}{k}-\int_0^\infty  {\frac{{e^{ - kt}-e^{ - xt} }}{{1 - e^{ - kt} }}} dt,
\end{align}

and
\begin{align}\label{1.3-eq}
\psi_k^{(m)}(x)&=(-1)^{m+1}m!\sum_{n=0}^{\infty}\frac{1}{(nk+x)^{m+1}}, m\geq1\\
 &= ( - 1)^{m + 1} \int_0^\infty  {\frac{1}{{1 - e^{ - kt} }}} t^m e^{ - xt} dt, m\geq1.
\end{align}
For the given complex numbers $a, b, c$ with $c\neq 0, -1,-2,\cdots$, the \textit{Gaussian hypergeometric function} is defined by
$$F(a,b;c;z)=~_{2}F_1(a,b;c;z)=\sum_{n\geqslant 0}\frac{(a,n)(b,n)}{(c,n)}\frac{z^n}{n!}, |z|<1.$$
Here $(a,0)=1$ for $(a\neq 0)$, and $(a,n)$ for $n\in \mathbb{N}$ is the shifted factorial or Appell symbol $(a,n)=a(a+1)\cdots(a+n-1).$
For more properties of these functions, the reader may see the references \cite{na,nmn,npt}.

In this paper, we will study three problems based on $k$-gamma and $k-$ digamma functions, mainly involving series expansions of the Furdui-type integral, the properties and inequalities of the Hadamard $k-$gamma function, and some properties of the $k$-Nielsen beta function. Therefore,
This paper is organized as follows: In Section 2, we list several useful lemmas. In Section 3, we will give four series expansions to the Furdui-type integral related Riemann zeta function and hypergeometric function. Section 4 gives some new identities and inequalities on the Hadamard $k$-gamma function. Section 5 presents several identities for the Nielsen $k$-beta function.

\section{Several Lemmas}

\begin{lemma}\rm{(\cite[Proposition 6]{dp})}\label{2.1-lem}
For $k,x>0$, we have
\begin{equation}
\Gamma _k (x) = k^{\frac{x}{k} - 1} \Gamma \left( {\frac{x}{k}} \right)
\end{equation}
and \begin{equation}
\Gamma _k (x)\Gamma _k (k-x)=\frac{\pi}{\sin\left(\frac{\pi x}{k}\right)}
\end{equation}
\end{lemma}

\begin{lemma}\rm{(\cite[Lemma 2.1]{yinhuang})}\label{2.2-lem}
For $k,x>0$, we have
\begin{equation}
\psi _k (x) = \frac{{\ln k}}{k} + \frac{{\psi (x /k)}}{k}.
\end{equation}
\end{lemma}

\begin{lemma}\rm{(\cite[3. 194. 1]{gr})}\label{2.3-lem}
Let $\Re a>0, \arg (1+bu)<\pi$, Then
\begin{equation}
\int_0^u {\frac{{{x^{a - 1}}}}{{{{(1 + bx)}^v}}}dx = \frac{{{u^a}}}{a}} F(v,a;1 + a; - bu).
\end{equation}
\end{lemma}

\begin{lemma}\label{2.4-lem}
For $k,x>0$, we have
\begin{equation}\label{2.5-eq}
\psi _k (x+k) =\psi_{k}(x)+\frac{1}{x}
\end{equation}
\end{lemma}
\begin{proof}
Take the logarithm of both sides of the formula \ref{1.1-eq} and then calculate the derivative.
\end{proof}

\begin{lemma}\rm{(\cite[Theorem 2.2]{zhang})}\label{2.5-lem}
For $x>0$ and $k>0$, The function $x\beta_{k}(x)$
is completely monotonic, decreasing and convex.
\end{lemma}

\begin{lemma}\label{2.6-lem}
For $x >0$ and $k>0$, the function $\beta_{k}(x)$ satisfies the following inequality:
\begin{eqnarray}\label{2.6-eq}
2[\beta^{\prime}_{k}(x)]^2-\beta^{\prime\prime}_{k}(x)\beta_{k}(x)>0\end{eqnarray}
\end{lemma}
\begin{proof}
The inequality \ref{2.6-eq} can be written as
\begin{equation*}
\frac{[\beta^{\prime}_{k}(x)]^2-\beta^{\prime\prime}_{k}(x)\beta_{k}(x)}{[\beta_{k}(x)]^2}>-\left[\frac{\beta^{\prime}_{k}(x)}{\beta_{k}(x)}\right]^2,
\end{equation*}
which is equal to
\begin{equation}
\left[\frac{\beta^{\prime}_{k}(x)}{\beta_{k}(x)}\right]^{\prime}<\left[\frac{\beta^{\prime}_{k}(x)}{\beta_{k}(x)}\right]^2\Leftrightarrow
\frac{\left[\beta^{\prime}_{k}(x)/\beta_{k}(x)\right]^{\prime}}{\left[\beta^{\prime}_{k}(x)/\beta_{k}(x)\right]^2}<1.  \label{autonomouce21301}
\end{equation}
 If formula (\ref{autonomouce21301}) holds, then integrating on both sides of the right inequality, we have
 \begin{align*}
 \int_\alpha^x \frac{\left[\beta^{\prime}_{k}(t)/\beta_{k}(t)\right]^{\prime}}{\left[\beta^{\prime}_{k}(t)/\beta_{k}(t)\right]^2}\mathrm{d}t
 =\frac{\beta_{k}(\alpha)}{\beta^{\prime}_{k}(\alpha)}-\frac{\beta_{k}(x)}{\beta^{\prime}_{k}(x)}<x-\alpha,~0<\alpha<x.
\end{align*}
Owing to \ref{5.11-eq} and the functions $\beta_{k}(k+\alpha)$ and $\beta^{\prime}_{k}(k+\alpha)$ are convergent, then
\begin{align*}
\lim_{\alpha\rightarrow 0^+}\frac{\beta_{k}(\alpha)}{\beta^{\prime}_{k}(\alpha)}=\lim_{\alpha\rightarrow 0^+}\frac{\frac{1}{\alpha}-\beta_{k}(k+\alpha)}{-\frac{1}{\alpha^2}-\beta^{\prime}_{k}(k+\alpha)}=0.
\end{align*}
This implies that
 \begin{align*}
 -\frac{\beta_{k}(x)}{\beta^{\prime}_{k}(x)}<x,~x>0.
 \end{align*}
Noting Lemma \ref{2.5-lem}, the proof is complete.
\end{proof}

\begin{lemma} \label{2.7-lem}
 For $k>0$, the function
$\lambda(x)=\frac{x\beta^{\prime}_{k}(x)}{\beta^2_{k}(x)}$
is decreasing on $x \in (0,\infty)$.
\end{lemma}

\begin{proof}By virtue of Lemma (\ref{2.5-lem}) and Lemma(\ref{2.6-lem}), we easily obtain
 \begin{eqnarray}
\lambda^{\prime}(x)=\frac{\beta^{\prime}_{k}(x)\beta^2_{k}(x)+x\left\{\beta^{\prime\prime}_{k}(x)\beta_{k}(x)
 -2[\beta^{\prime}_{k}(x)
 ]^2\right\}\beta_{k}(x)}{\beta^4_{k}(x)}<0,\label{autonomouce2114}
\end{eqnarray}
which implies that $\lambda(x)$ is decreasing.
\end{proof}

\begin{remark}
The function $\beta_{k}(x)$ in Lemma \ref{2.5-lem}-Lemma \ref{2.7-lem} is called Nielsen $k$-beta function. The detailed definition may see Section 5.
\end{remark}

\section{Series expansions of Furdui-type integral}

In 2014, Furdui proposed an open problem 103 in Volume 4, Issue 3 of Mathproblems: \\
\textbf{Open Problem.} Calculate, if possible, in terms of the well-known constants the integral $\int_0^1 {{x^m}{\psi}(x)}dx$ where $k\geq3$ is an integer. If $m=2$, he obtained the formula $\int_0^1 {{x^2}{\psi}(x)}dx=\ln\left(\frac{A}{\sqrt{2\pi}}\right),$ where the constant $A$ is Glaisher-Kinkelin constant.
Based on Riemann zeta function and hypergeometric function, we calculate more generalized Furdui integral, and present several series expansions formulas for this type of integral.
\begin{theorem}\label{3.1-thm}
For $k>0,m\in \mathbb{{N}}$, we have
\begin{equation}\label{(3.1-eq)}
\int_0^k {{x^m}{\psi _k}(x)dx = \frac{{{k^m}(\ln k - \gamma )}}{{m +
1}}}  - \frac{{{k^m}}}{m} + {k^m}\sum\limits_{s = 2}^\infty
{\frac{{{{( - 1)}^s}}}{{m + s}}} \zeta (s).
\end{equation}
\end{theorem}
\begin{proof}
Applying Lemma \ref{2.2-lem} and the identities
\begin{equation*}
\psi(x+1)=\psi(x)+\frac{1}{x}
\end{equation*}
\begin{equation}\label{(3.2-eq)}
\psi (x + 1) =  - \gamma  + \sum\limits_{s = 2}^\infty  {{{( - 1)}^s}} \zeta (s){x^{s - 1}},
\end{equation}
(The formula \eqref{(3.2-eq)} may see formula 8. 363. 1 in reference \cite{gr})
we may obtain
$$\begin{array}{l}
\int_0^k {{x^m}{\psi _k}(x)dx}  = \frac{{{k^m}\ln k}}{{m + 1}} + \frac{1}{k}\int_0^k {{x^m}\psi \left( {\frac{x}{k}} \right)dx} \\
 = \frac{{{k^m}\ln k}}{{m + 1}} + {k^m}\int_0^1 {{t^m}\psi \left( t \right)dt} \\
 = \frac{{{k^m}(\ln k - \gamma )}}{{m + 1}} - \frac{{{k^m}}}{m} + {k^m}\sum\limits_{s = 2}^\infty  {\frac{{{{( - 1)}^s}}}{{m + s}}} \zeta (s).
\end{array}$$
The proof is complete.
\end{proof}
\begin{remark}
Here, we give a new proof to Theorem \ref{3.1-thm}. In fact, applying the formula \ref{1.2-eq}, we have
$$\begin{array}{l}
\int_0^k {{x^m}{\psi _k}(x)dx}  = \int_0^k {\frac{{\ln k - \gamma }}{k}{x^m}dx}  - \int_0^k {{x^{m - 1}}dx}  + \sum\limits_{n = 1}^\infty  {\int_0^k {\frac{{{x^{m + 1}}}}{{nk(nk + x)}}dx} } \\
 = \frac{{{k^m}(\ln k - \gamma )}}{{m + 1}} - \frac{{{k^m}}}{m} + \sum\limits_{n = 1}^\infty  {\int_0^k {\frac{{{x^{m + 1}} - {{( - nk)}^{m + 1}} + {{( - nk)}^{m + 1}}}}{{nk(nk + x)}}dx} } \\
 = \frac{{{k^m}(\ln k - \gamma )}}{{m + 1}} - \frac{{{k^m}}}{m} + \sum\limits_{n = 1}^\infty  {\left( {\sum\limits_{j = 0}^m {{{( - 1)}^j}{{(nk)}^{j - 1}}\int_0^k {{x^{m - j}}dx}  + {{( - 1)}^{m + 1}}{{(nk)}^m}\int_0^k {\frac{1}{{nk + x}}dx} } } \right)} \\
 = \frac{{{k^m}(\ln k - \gamma )}}{{m + 1}} - \frac{{{k^m}}}{m} + \sum\limits_{n = 1}^\infty  {\left( {\sum\limits_{j = 0}^m {\frac{{{{( - 1)}^j}{n^{j - 1}}{k^m}}}{{m - j + 1}} + {{( - 1)}^{m + 1}}{n^m}{k^m}\ln \left( {1 + \frac{1}{n}} \right)} } \right)} \\
 = \frac{{{k^m}(\ln k - \gamma )}}{{m + 1}} - \frac{{{k^m}}}{m} + {k^m}\sum\limits_{n = 1}^\infty  {{{( - 1)}^{m + 1}}\left( {\sum\limits_{j = 0}^m {\frac{{{{( - 1)}^{j - m - 1}}{n^{j - 1}}}}{{m - j + 1}} + \sum\limits_{j = 1}^\infty  {\frac{{{{( - 1)}^{j + 1}}}}{{j{n^{j - m}}}}} } } \right)} \\
 = \frac{{{k^m}(\ln k - \gamma )}}{{m + 1}} - \frac{{{k^m}}}{m} + {k^m}\sum\limits_{n = 1}^\infty  {{{( - 1)}^{m + 1}}\sum\limits_{j = m + 2}^\infty  {\frac{{{{( - 1)}^{j + 1}}}}{{j{n^{j - m}}}}} } \\
 = \frac{{{k^m}(\ln k - \gamma )}}{{m + 1}} - \frac{{{k^m}}}{m} + \sum\limits_{j = m + 2}^\infty  {{{( - 1)}^{m + 1}}{k^m}\frac{{{{( - 1)}^{j + 1}}}}{j}} \sum\limits_{n = 1}^\infty  {\frac{1}{{{n^{j - m}}}}} \\
 = \frac{{{k^m}(\ln k - \gamma )}}{{m + 1}} - \frac{{{k^m}}}{m} + {k^m}\sum\limits_{s = 2}^\infty  {\frac{{{{( - 1)}^s}}}{{m + s}}} \zeta (s).
\end{array}$$

\end{remark}

\begin{theorem}\label{3.2-thm}
For $k>0,m\in \mathbb{{N}}$, we have
\begin{equation}\label{(3.3-eq)}
\int_0^k {{x^m}{\psi _k}(x)dx}  = \frac{{{k^m}(\ln k - m\gamma )}}{{m + 1}} - \frac{{{k^m}}}{m} + m{k^m}\sum\limits_{s = 2}^\infty  {\frac{{{{( - 1)}^{s + 1}}}}{{s(m + s)}}} \zeta (s).
\end{equation}
\end{theorem}
\begin{proof}
By using the series expansion of the function $\ln\Gamma(x+1)$(See the formula 8. 342. 1 in reference \cite{gr}) and Lemma \ref{2.1-lem}, we easily get $$\begin{array}{l}
\ln {\Gamma _k}(x) =\left( {\frac{x}{k} - 1} \right)\ln k + \ln \Gamma \left( {\frac{x}{k}} \right)\\
 = \frac{{(\ln k - \gamma )x}}{k} - \ln x+ \sum\limits_{s = 2}^\infty  {\frac{{{{( - 1)}^s}\zeta (s)}}{s}{{\left( {\frac{x}{k}} \right)}^s}.}
\end{array}$$
In addition, we have
\begin{equation}\label{3.4-eq}
\int_0^k {{x^{m - 1}}\ln {\Gamma _k}(x)dx}  = \frac{{{k^m}(\ln k - \gamma )}}{{m + 1}} - \frac{{{k^m}{\mathop{ \ln k}\nolimits} }}{m} + \frac{{{k^m}}}{{{m^2}}} + \sum\limits_{s = 2}^\infty  {\frac{{{{( - 1)}^s}{k^m}}}{{s(m + s)}}} \zeta (s).
\end{equation}

Thus, applying integration by part, we have
\begin{equation}\label{3.5-eq}
 \int_0^k {{x^m}{\psi _k}(x)dx}  = \int_0^k {{x^m}d\ln {\Gamma _k}(x) =  - } m\int_0^k {{x^{m - 1}}\ln {\Gamma _k}(x)dx}.
\end{equation}
Combining \eqref{3.4-eq} with \eqref{3.5-eq}, we can complete the proof.

\end{proof}

\begin{theorem}\label{3.3-thm}
For $k>0,m\in \mathbb{{N}}$, we have
\begin{equation}\label{(3.6-eq)}
\begin{array}{l}
\int_0^k {{x^m}{\psi _k}(x)dx}  = \frac{{ - m{k^m}(\ln k - \gamma )}}{{m + 1}} + \frac{{3m}}{2}\left( {\frac{{{k^m}\ln k}}{m} - \frac{{{k^m}}}{{{m^2}}}} \right) + \frac{{{k^m}\ln \left( {\frac{\pi }{k}} \right)}}{2}\\
 + \frac{{m{k^m}}}{{2{\pi ^m}}}\int_0^\pi  {{x^{m - 1}}\ln \sin xdx}  + m\sum\limits_{n = 1}^\infty  {\frac{{{k^m}\zeta (2n + 1)}}{{(2n + 1)(2n + m + 1)}}} .
\end{array}
\end{equation}
\end{theorem}
\begin{proof}
By using the series expansion of $\ln\Gamma(x+1)$(See the formula 8. 342. 2 in reference \cite{gr}), we may get
\begin{equation}\label{3.7-eq}
\ln {\Gamma _k}(x) = \frac{{(\ln k - \gamma )x}}{k} - \frac{3}{2}\ln x - \frac{1}{2}\ln \left( {\frac{\pi }{k}} \right) - \frac{1}{2}\ln \left( {\sin \frac{{\pi x}}{k}} \right) - \sum\limits_{n = 1}^\infty  {\frac{{\zeta (2n + 1)}}{{(2n + 1)}}{{\left( {\frac{x}{k}} \right)}^{2n + 1}}} .
\end{equation}
Substitute the formula \ref{3.7-eq} to \ref{3.8-eq}
\begin{equation}\label{3.8-eq}
\int_0^k {{x^m}{\psi _k}(x)dx}  =  - m\int_0^k {{x^{m - 1}}\ln {\Gamma _k}(x)dx},
\end{equation}
we can complete the proof.
\end{proof}

\begin{theorem}\label{3.4-thm}
For $k>0,m\in \mathbb{{N}}$, we have
\begin{equation}\label{(3.9-eq)}
\begin{array}{l}
\int_0^k {{x^m}{\psi _k}(x)dx}  = \frac{{{k^{m + 1}}}}{{m + 1}}{\psi _k}(k) - \frac{{{k^{m + 2}}}}{{(m + 1)(m + 2)}}{{\psi '}_k}(k) +  \cdots  + \frac{{{{( - 1)}^{n - 1}}{k^{m + n}}}}{{(m + 1)(m + 2) \cdots (m + n)}}\psi _k^{(n - 1)}(k)\\
 + \frac{{{{( - 1)}^{n + 1}}{k^m}n!}}{m} - \sum\limits_{i = 1}^\infty  {\frac{{n!{k^m}F\left( {n + 1,m + n + 1;m + n + 2; - \frac{1}{i}} \right)}}{{(m + 1)(m + 2) \cdots (m + n + 1){i^{n + 1}}}}} .
\end{array}
\end{equation}
\end{theorem}
\begin{proof}
Applying integration by part, we have
$$\begin{array}{l}
\int_0^k {{x^m}{\psi _k}(x)dx}  = \int_0^k {{\psi _k}(x)d\left( {\frac{{{x^{m + 1}}}}{{m + 1}}} \right)}  = \frac{{{k^{m + 1}}}}{{m + 1}}{\psi _k}(k) - \frac{1}{{m + 1}}\int_0^k {{x^{m + 1}}{{\psi '}_k}(x)dx} \\
 = \frac{{{k^{m + 1}}}}{{m + 1}}{\psi _k}(k) - \frac{{{k^{m + 2}}}}{{(m + 1)(m + 2)}}{{\psi '}_k}(k) + \frac{1}{{(m + 1)(m + 2)}}\int_0^k {{x^{m + 2}}{{\psi ''}_k}(x)dx.}
\end{array}$$
Proceeding in sequence, we get
$$\begin{array}{l}
\int_0^k {{x^m}{\psi _k}(x)dx}  = \frac{{{k^{m + 1}}}}{{m + 1}}{\psi _k}(k) - \frac{{{k^{m + 2}}}}{{(m + 1)(m + 2)}}{{\psi '}_k}(k) +  \cdots  + \\
\frac{{{{( - 1)}^{n - 1}}{k^{m + n}}}}{{(m + 1)(m + 2) \cdots (m + n)}}\psi _k^{(n - 1)}(k) + \frac{{{{( - 1)}^n}}}{{(m + 1)(m + 2) \cdots (m + n)}}\int_0^k {{x^{m + n}}\psi _k^{(n)}(x)dx.}
\end{array}$$
Applying formula \ref{1.3-eq}, we have $$\begin{array}{l}
\int_0^k {{x^{m + n}}\psi _k^{(n)}(x)dx}  = \int_0^k {{x^{m + n}}{{( - 1)}^{n + 1}}n!\sum\limits_{i = 0}^\infty  {\frac{1}{{{{(x + ik)}^{n + 1}}}}} dx} \\
 = \frac{{{{( - 1)}^{n + 1}}{k^m}n!}}{m} + \sum\limits_{i = 1}^\infty  {{{( - 1)}^{n + 1}}n!\int_0^k {\frac{{{x^{m + n}}}}{{{{(x + ik)}^{n + 1}}}}dx} } .
\end{array}$$
For the last integral in the above equation, we change the form $$\int_0^k {\frac{{{x^{m + n}}}}{{{{(x + ik)}^{n + 1}}}}dx}  = \frac{1}{{{{(ik)}^{n + 1}}}}\int_0^k {\frac{{{x^{m + n}}}}{{{{(\frac{x}{{ik}} + 1)}^{n + 1}}}}dx} .$$
Next, putting $u=k, a=m+n+1, b=\frac{1}{ik}, v=n+1$ in Lemma \ref{2.3-lem}, we get $$\int_0^k {\frac{{{x^{m + n}}}}{{{{(\frac{x}{{ik}} + 1)}^{n + 1}}}}dx}  = \frac{{{k^{m + n + 1}}}}{{m + n + 1}}F\left( {n + 1,m + n + 1;m + n + 2; - \frac{1}{i}} \right).$$
Integrating the above results, we can complete the proof.
\end{proof}

\begin{remark}
Setting $n=1$ in Theorem \ref{3.4-thm}, we get
\begin{equation}\label{3.10-eq}
\int_0^k {{x^m}{\psi _k}(x)dx}  = \frac{{{k^{m + 1}}}}{{m + 1}}{\psi _k}(k) + \frac{{{k^m}}}{m} - \sum\limits_{i = 1}^\infty  {\frac{{{k^m}F\left( {2,m + 2;m +3; - \frac{1}{i}} \right)}}{{(m + 1)(m + 2)){i^2}}}}.
\end{equation}
\end{remark}

\section{Identities and inequalities on the Hadamard $k$-gamma function}
In 1894, French Mathematician Hadamard posed a function, which coincide with the gamma function at the natural numbers, but possesses no singularities in the complex plane
\begin{equation}\label{4.1-eq}
H(x) = \frac{1}{{\Gamma (1 - x)}}\frac{d}{{dx}}\ln \left( {\frac{{\Gamma \left( {\frac{1}{2} - \frac{x}{2}} \right)}}{{\Gamma \left( {1 - \frac{x}{2}} \right)}}} \right)
\end{equation}
Luschny\cite{lus} showed that $H(x)$ can be expressed:
\begin{equation}\label{4.2-eq}
H(x) = \Gamma (x)\left[ {1 + \frac{{\sin \left( {\pi x} \right)}}{{2\pi }}\left( {\psi \left( {\frac{x}{2}} \right) - \psi \left( {\frac{{x + 1}}{2}} \right)} \right)} \right].
\end{equation}
Hadamard's gamma function satisfies the following functional equation:
\begin{equation}\label{4.3-eq}
H(x + 1) = xH(x) + \frac{1}{{\Gamma (1 - x)}}.
\end{equation}
It is worth noting that Luschny\cite{lus} and Newton\cite{new} detailed studied the function, which are closely related to Newton's function. The functional equation \ref{4.3-eq} also has been investigated by Mijajlovi\'{c} and Male\v{s}evi\'{c} in \cite{mmm}.
In 2009, Alzer\cite{alzer} presented an elegant inequality of Hadamard's gamma function: The inequality
\begin{equation}\label{4.4-eq}
H(x)+H(y)\leq H(x+y)
\end{equation}
holds for all real numbers $x,y\geq \alpha$ if and only if $\alpha\geq \alpha_0$. Here, $\alpha_0$ is the only solution of $H(2t)=2H(t)$ in $[1.5, \infty)$.

Based on $\Gamma_{k}(x)$, we shall give a new $k-$generalization of Hadamard's gamma function, and further study the properties of this new function.
\begin{definition}
For all real numbers, we define the Hadamard $k-$gamma function as follows:
\begin{equation}\label{4.5-eq}
\begin{array}{l}
{H_k}(x) = \frac{1}{{{\Gamma _k}(k - x)}}\frac{d}{{dx}}\ln \left( {\frac{{\Gamma \left( {\frac{k}{2} - \frac{x}{2}} \right)}}{{\Gamma \left( {k - \frac{x}{2}} \right)}}} \right)\\
 = \frac{{{\psi _k}\left( {k - \frac{x}{2}} \right) - {\psi _k}\left( {\frac{k}{2} - \frac{x}{2}} \right)}}{{2{\Gamma _k}(k - x)}}.
\end{array}
\end{equation}
\end{definition}
\begin{remark}
If $k=1$, the Hadamard $k-$gamma function transforms into the Hadamard gamma function.
\end{remark}

\begin{theorem}\label{4.1-thm}
For $k>0$, we have
\begin{equation}\label{4.6-eq}
{H_k}(x + k) = x{H_k}(x) + \frac{1}{{{\Gamma _k}(k - x)}}.
\end{equation}
\end{theorem}
\begin{proof}
By using formula \ref{1.1-eq} and \ref{2.5-eq}, we get
$$\begin{array}{l}
{H_k}(x + k) = \frac{{{\psi _k}\left( {\frac{k}{2} - \frac{x}{2}} \right) - {\psi _k}\left( { - \frac{x}{2}} \right)}}{{2{\Gamma _k}( - x)}}\\
 = \frac{{x\left[ {{\psi _k}\left( {k - \frac{x}{2}} \right) - {\psi _k}\left( {\frac{k}{2} - \frac{x}{2}} \right)} \right]}}{{2{\Gamma _k}(k - x)}} + \frac{1}{{{\Gamma _k}(k - x)}}\\
 = x{H_k}(x) + \frac{1}{{{\Gamma _k}(k - x)}}.
\end{array}$$

\end{proof}
\begin{remark}
Putting $x=0$ in formula \ref{4.6-eq}, we get $H_{k}(k)=\frac{1}{\Gamma_{k}(k)}=1.$
\end{remark}
Based on Theorem \ref{4.1-thm}, it is easy to obtain the following corollary through recursion.
\begin{corollary}
For $k>0,n\in \mathbb{{N}}$, we have
\begin{equation}\label{4.7-eq}
\begin{array}{l}
{H_k}(x + nk) = \left[ {x + (n - 1)k} \right]\left[ {x + (n - 2)k} \right] \cdots (x+1)x{H_k}(x)\\
 + \frac{1}{{{\Gamma _k}\left[ {(2 - n)k - x} \right]}} + \frac{{x + (n - 1)k}}{{{\Gamma _k}\left[ {(3 - n)k - x} \right]}} +  \cdots + \frac{{\left[ {x + (n - 1)k} \right]\left[ {x + (n - 2)k} \right] \cdots (x + 1)x}}{{{\Gamma _k}(k - x)}}.
\end{array}
\end{equation}
\end{corollary}

\begin{theorem}\label{4.2-thm}
For $k>0$ and $x>0$, we have
\begin{equation}\label{4.8-eq}
{H_k}(x) = \frac{{{\Gamma _k}(x)}}{k} - \frac{{{\Gamma _k}(x)\sin\left( {\frac{{\pi x}}{k}} \right)}}{\pi }{\beta _k}(x),
\end{equation}
where the function ${\beta _k}(x)=\frac{1}{2}\left( {{\psi _k}\left( {\frac{{x + k}}{2}} \right) - {\psi _k}\left( {\frac{x}{2}} \right)} \right)$ is called Nielsen $k-$beta function.
\end{theorem}
\begin{proof}
By using the formula \ref{2.2-lem}, we have
$${\Gamma _k}\left( {\frac{k}{2} + x} \right){\Gamma _k}\left( {\frac{k}{2} - x} \right) = \frac{\pi }{{\cos \left( {\frac{{2x}}{k}} \right)}}.$$
Taking the logarithm and then differentiating, we get $${\psi _k}\left( {\frac{k}{2} + x} \right) - {\psi _k}\left( {\frac{k}{2} - x} \right) = \frac{\pi }{k}\tan \left( {\frac{{\pi x}}{k}} \right).$$
Similarly, we also get $${\psi _k}\left( {\frac{x}{2}} \right) + {\psi _k}\left( {k - \frac{x}{2}} \right) =  - \frac{\pi }{k}\cot \left( {\frac{{\pi x}}{{2k}}} \right).$$
Furthermore, we have
$$\begin{array}{l}
{H_k}(x) = \frac{{{\Gamma _k}(x)\sin\left( {\frac{{\pi x}}{k}} \right)}}{{2\pi }}\left[ {{\psi _k}\left( {k - \frac{x}{2}} \right) - {\psi _k}\left( {\frac{{k - x}}{2}} \right)} \right]\\
 = \frac{{{\Gamma _k}(x)\sin\left( {\frac{{\pi x}}{k}} \right)}}{{2\pi }}\left[ {{\psi _k}\left( {\frac{x}{2}} \right) + \frac{\pi }{k}\cot \left( {\frac{{\pi x}}{{2k}}} \right) - {\psi _k}\left( {\frac{{k + x}}{2}} \right) + \frac{\pi }{k}\tan \left( {\frac{{\pi x}}{{2k}}} \right)} \right]\\
 = \frac{{{\Gamma _k}(x)}}{k} - \frac{{{\Gamma _k}(x)\sin\left( {\frac{{\pi x}}{k}} \right)}}{\pi }{\beta _k}(x).
\end{array}$$

\end{proof}

\begin{theorem}\label{4.3-thm}
The inequality
\begin{equation}\label{4.9-eq}
k^{y/k}H_{k}(x)+k^{x/k}H_{k}(y)\leq H_{k}(x+y)
\end{equation}
holds for all real numbers $x,y\geq \alpha$ and $k>0$ if and only if $\alpha\geq \alpha_0$. Here, $\alpha_0$ is the only solution of $H_k(2x)=2 k^{x/k}H_{k}(x)$ in $[1.5k, \infty)$.
\end{theorem}
\begin{proof}
Simple computation yields ${H_k}(x) = {k^{\frac{x}{k} - 1}}{H}\left( {\frac{x}{k}} \right)$. Noting Alzer inequality \ref{4.4-eq}, we can complete the proof.
\end{proof}

\begin{theorem}\label{4.4-thm}
For all real number $\left| x \right| < 1$,
\begin{equation}\label{4.10-eq}
2x\Phi(-1,1,-x)=\Phi\left(1,1,1-\frac{x}{2}\right)-\Phi\left(1,1,\frac{1}{2}-\frac{x}{2}\right),
\end{equation}
where $\Phi(z,s,a)=\sum\limits_{n = 0}^\infty  {\frac{{{z^n}}}{{{{(n + a)}^s}}}} $ is Lerch zeta function.
\end{theorem}
\begin{proof}
On one hand, we have (See Wikipedia Hadamard gamma function)$$
{H_k}(x) = {k^{\frac{x}{k} - 1}}{H}\left( {\frac{x}{k}} \right) = {k^{\frac{x}{k} - 1}}\frac{{\Phi \left( { - 1,1, - \frac{x}{k}} \right)}}{{\Gamma \left( { - \frac{x}{k}} \right)}}.$$
On the other hand, by using identities of integral $$\Phi \left( {z,s,a} \right) = \frac{1}{{\Gamma (s)}}\int_0^\infty  {\frac{{{t^{s - 1}}{e^{ - at}}}}{{1 - z{e^{ - t}}}}dt}$$ and formula (1.3), we also obtain $${H_k}(x) = \frac{1}{k}\left[ {\frac{{\Phi \left( {1,1,\frac{1}{2} - \frac{x}{{2k}}} \right) - \Phi \left( {1,1,1 - \frac{x}{{2k}}} \right)}}{{ - 2x{\Gamma _k}( - x)}}} \right].$$
Furthermore, we have
$$\frac{{2x}}{k}\Phi \left( { - 1,1, - \frac{x}{k}} \right) = \Phi \left( {1,1,\frac{1}{2} - \frac{x}{{2k}}} \right) - \Phi \left( {1,1,1 - \frac{x}{{2k}}} \right).$$
The proof is complete.

\end{proof}

\section{Several identities for the Nielsen $k$-beta function}
The Nielsen's $\beta$-function is defined as (\cite{kanat, nni})
    \begin{align}
\beta(x)=\int_0^1 \frac{t^{x-1}}{1+t}dt=\int_0^{\infty} \frac{e^{-xt}}{1+e^{-t}}dt
=\sum^{\infty}_{n=0}\frac{(-1)^n}{n+x}
=\frac{1}{2}\left\{ \psi \left(\frac{x+1}{2}\right)-\psi \left(\frac{x}{2}\right)\right\}
  \end{align}
where $x \in (0,\infty).$
 The function can be used to calculate some integrals(\cite{knamc}). Recently, K. Nantomah studied the properties and inequalities of a $p-$generalization of the Nielsen's function in \cite{kanat}.
In 2019, Zhang et. al.\cite{zhang} discussed a new $k$-generalization of the Nielsen's $\beta$-function.  Later, they studied the completely monotonicity, convexity and inequalities of the new function. The new function has been called Nielsen $k-$ beta function. Its definition is as follows:
\begin{align}\label{5.2-eq}
\beta_k(x)&=\int_0^1 \frac{t^{x-1}}{1+t^k}dt
=\int_0^{\infty} \frac{e^{-xt}}{1+e^{-kt}}dt\\
&=\sum^{\infty}_{n=0}\left(\frac{1}{2nk+x}-\frac{1}{2nk+k+x}\right)
=\frac{1}{2}\left\{ \psi_k \left(\frac{x+k}{2}\right)-\psi_k \left(\frac{x}{2}\right)\right\}
  \end{align}
where $k>0,x \in (0,\infty).$

Next, we will mainly study some series expansions of this function, which are based on the Riemann zeta function.
\begin{theorem}\label{5.1-thm}
For $x>0$ and $n\in \mathbb{N}$, we have
\begin{equation}\label{5.4-eq}
\sum\limits_{m = 1}^n {{\beta _k}\left[ {{{(2k)}^m}x} \right] = } {\psi _k}\left( {{2^n}{k^n}x} \right) - {\psi _k}\left( {kx} \right) - \frac{{n\ln 2}}{k}.
\end{equation}
\end{theorem}
\begin{proof}
Applying the Legendre relation $$\Gamma (2x) = \frac{{{2^{2x - 1}}}}{{\sqrt \pi  }}\Gamma (x)\Gamma \left( {x + \frac{1}{2}} \right)$$ and Lemma \ref{2.1-lem}, we get
\begin{equation}\label{5.5-eq}
{\Gamma _k}(2kx) = \frac{{{2^{2x - 1}}}}{{\sqrt {k\pi } }}{\Gamma _k}(kx){\Gamma _k}\left( {kx + \frac{k}{2}} \right).
\end{equation}
Taking the logarithm of both sides of the formula \ref{5.5-eq} and differentiating, we have
\begin{equation}\label{5.55-eq}
{\psi _k}\left( {kx + \frac{k}{2}} \right) = 2{\psi _k}\left( {2kx} \right) - {\psi _k}\left( {kx} \right) - \frac{{2\ln 2}}{k}.
\end{equation}
Furthermore, we can get
$$\begin{array}{l}
\sum\limits_{m = 1}^n {{\beta _k}\left[ {{{(2k)}^m}x} \right] = } \frac{1}{2}{\psi _k}\left( {\frac{{2kx + k}}{2}} \right) - \frac{1}{2}{\psi _k}\left( {kx} \right) + \frac{1}{2}{\psi _k}\left( {\frac{{4kx + k}}{2}} \right)\\
 - \frac{1}{2}{\psi _k}\left( {2kx} \right) + \; \cdots  + \frac{1}{2}{\psi _k}\left( {\frac{{{{(2k)}^n}x + k}}{2}} \right) - \frac{1}{2}{\psi _k}\left( {{{(2k)}^{n - 1}}x} \right)\\
\end{array}$$
Substitute Formula \ref{5.55-eq} into the above equation, and then the proof of the Theorem \ref{5.1-thm} can be completed.
\end{proof}
\begin{theorem}\label{5.2-thm}
For $x>0$ and $k>0$, we have
\begin{equation}\label{5.6-eq}
{\beta _k}(x) = \sum\limits_{n = 0}^\infty  {\frac{{{{( - 1)}^n}}}{{x + nk}}.}
\end{equation}
\end{theorem}
\begin{proof}
The series expansion $\frac{1}{{1 + x}} = \sum\limits_{n = 0}^\infty  {{{( - 1)}^n}{x^n}} $ can result in
$${\beta _k}(x) = \int_0^1 {\frac{{{t^{x - 1}}}}{{1 + {t^k}}}dt = \sum\limits_{n = 0}^\infty  {{{( - 1)}^n}\int_0^1 {{t^{x + kn - 1}}dt} }  = } \sum\limits_{n = 0}^\infty  {\frac{{{{( - 1)}^n}}}{{x + nk}}.} $$

\end{proof}

\begin{theorem}\label{5.3-thm}
For $x>0$ and $k>0$, we have
\begin{equation}\label{5.7-eq}
{\beta _k}\left( {\frac{{x + k}}{2}} \right) = \int_0^\infty  {\frac{{{e^{ - xt}}}}{{\cosh (kt)}}dt = } \int_0^\infty  {\frac{{2{e^{ - xt}}}}{{{e^{kt}} + {e^{ - kt}}}}dt.}
\end{equation}
\end{theorem}
\begin{proof}
On one hand, we easily obtain $${\beta _k}\left( {\frac{{x + k}}{2}} \right) = 2\sum\limits_{n = 0}^\infty  {\frac{{{{( - 1)}^n}}}{{x + 2nk + k}}.} $$ by using Theorem \ref{5.2-thm}. On the other hand, the substitution $u=e^{-kt}$ yields
$$\begin{array}{l}
\int_0^\infty  {\frac{{2{e^{ - xt}}}}{{{e^{kt}} + {e^{ - kt}}}}dt}  = \frac{2}{k}\int_0^1 {\frac{{{u^{x/k}}}}{{1 + {u^2}}}du} \\
 = \frac{2}{k}\int_0^1 {{u^{x/k}}\sum\limits_{n = 0}^\infty  {{{( - 1)}^n}{u^{2n}}} du}  = 2\sum\limits_{n = 0}^\infty  {\frac{{{{( - 1)}^n}}}{{x + 2nk + k}}.}
\end{array}$$
The proof is complete.
\end{proof}

\begin{theorem}\label{5.4-thm}
For $|x|<k$ and $k>0$, we have
\begin{equation}\label{5.8-eq}
{\beta _k}\left( {x + k} \right) = \frac{{\ln 2}}{k} + \sum\limits_{m = 1}^\infty  {\frac{{{{( - 1)}^m}(1 - {2^{ - m}})\zeta (m + 1)}}{{{k^{m + 1}}}}{x^m}.}
\end{equation}
\end{theorem}
\begin{proof}
By using the Taylor formula, we have $${\beta _k}\left( {x + k} \right) = {\beta _k}\left( k \right) + \sum\limits_{m = 1}^\infty  {\frac{{\beta _k^{(m)}\left( k \right)}}{{m!}}{x^m} = \frac{{\ln 2}}{k} + \sum\limits_{m = 1}^\infty  {\sum\limits_{n = 0}^\infty  {\frac{{{{( - 1)}^{m + n}}}}{{{{(k + nk)}^{m + 1}}}}} {x^m}} .} $$

Considering to the identity $$
\sum\limits_{n = 0}^\infty  {\frac{{{{( - 1)}^n}}}{{{{(1 + n)}^{m + 1}}}}}  = \sum\limits_{n = 1}^\infty  {\frac{{{{( - 1)}^n}}}{{{n^{m + 1}}}}}  - \frac{1}{{{2^m}}}\sum\limits_{n = 1}^\infty  {\frac{1}{{{n^m}}}}  = (1 - {2^{ - m}})\zeta (m + 1),$$ we may get the series expansion \ref{5.8-eq}.
By applying the Cauchy criterion, we obtain that the convergence domain of the series is $|x|<k$.
\end{proof}
\begin{remark}
By using the identity\cite[Theorem 2.2]{zhang}
\begin{equation}\label{5.11-eq}
 {\beta _k}\left( {x + k} \right)+{\beta _k}\left( {x} \right)=\frac{1}{x},
\end{equation}
so, we may obtained a power series expansion of $x{\beta _k}\left( {x} \right)$ as follows:
$$x{\beta _k}\left( {x } \right) = 1-\frac{{x\ln 2}}{k} + \sum\limits_{m = 1}^\infty  {\frac{{{{( - 1)}^{m+1}}(1 - {2^{ - m}})\zeta (m + 1)}}{{{k^{m + 1}}}}{x^{m+1}}.}$$
This is an alternative series. So we can obtain the following inequalities: For $0<x<k,$ we have
$$\frac{1}{x}-\frac{\ln2}{k}<\beta _{k} (x)<\frac{1}{x}$$ and
$$\frac{1}{x}-\frac{\ln2}{k}<\beta _{k} (x)<\frac{1}{x}-\frac{\ln2}{k}+\frac{\pi^{2}x}{12k^2}.$$

\end{remark}
\begin{theorem}\label{5.5-thm}
For $x\neq 0$ and $k>0$, we have
\begin{equation}\label{5.9-eq}
{\beta _k}\left( x \right) = \frac{1}{x} - \frac{1}{{x + k}} + \sum\limits_{n = 1}^\infty  {\sum\limits_{i = 0}^{n - 1} {\frac{{{{( - 1)}^{n + 1}}\left( \begin{array}{l}
n\\
i
\end{array} \right)\zeta (n + 1)}}{{{2^{n + 1}}{k^{i + 1}}}}} {x^i}} .
\end{equation}
\end{theorem}
\begin{proof}
Applying the identity 8. 363. 1 $$\psi (x + 1) =  - \gamma  + \sum\limits_{m = 2}^\infty  {{{( - 1)}^m}\zeta (m){x^{m - 1}}}$$ in reference \cite{gr}, we have $${\psi _k}(x) = \frac{1}{k}\psi \left( {\frac{x}{k}} \right) + \frac{{\ln k}}{k} = \frac{{\ln k - \gamma }}{k} - \frac{1}{x} + \sum\limits_{m = 1}^\infty  {\frac{{{{( - 1)}^{m + 1}}\zeta (m + 1)}}{{{k^{m + 1}}}}{x^m}.} $$
Therefore, we have $$\begin{array}{l}
{\beta _k}\left( x \right) = \frac{1}{2}\left[ {{\psi _k}\left( {\frac{{x + k}}{2}} \right) - {\psi _k}\left( {\frac{x}{2}} \right)} \right]\\
 = \frac{1}{x} - \frac{1}{{x + k}} + \sum\limits_{n = 1}^\infty  {\frac{{{{( - 1)}^{n + 1}}\zeta (n + 1)}}{{2{k^{n + 1}}}}\left[ {{{\left( {\frac{{x + k}}{2}} \right)}^n} - {{\left( {\frac{x}{2}} \right)}^n}} \right]} \\
 = \frac{1}{x} - \frac{1}{{x + k}} + \sum\limits_{n = 1}^\infty  {\sum\limits_{i = 0}^{n - 1} {\frac{{{{( - 1)}^{n + 1}}\left( \begin{array}{l}
n\\
i
\end{array} \right)\zeta (n + 1)}}{{{2^{n + 1}}{k^{i + 1}}}}} {x^i}} .
\end{array}$$
The proof is complete.
\end{proof}

\begin{theorem}\label{5.6-thm}
For $x>0$ and $k>0$, the function $\beta_{k}(x)$ satisfies
\begin{align}
 \frac{2\beta_{k}(x)\beta_{k}(k^2/x)}{\beta_{k}(x)+\beta_{k}(k^2/x)}\leq \beta_{k}(k)=\frac{\ln2}{k}, \label{autonomouce214}
\end{align}
where the equality is right for $x=k$.
\end{theorem}
\begin{proof}
The case for equality is trivial. Let $\mu(x)=\ln\left(\frac{2\beta_{k}(x)\beta_{k}(k^2/x)}{\beta_{k}(x)+\beta_{k}(k^2/x)}\right).$
Then direct computations result in \begin{align}
\mu^{\prime}(x)=\frac{\beta^{\prime}_{k}(x)}{\beta_{k}(x)}-\frac{k^2}{x^2}\frac{\beta^{\prime}_{k}(k^2/x)}{\beta_{k}(k^2/x)}
-\frac{\beta^{\prime}_{k}(x)-\frac{k^2}{x^2}\beta_{k}(k^2/x)}{\beta_{k}(x)+\beta_{k}(k^2/x)}
\end{align}
which indicates that
\begin{align}
x\left[\frac{1}{\beta_{k}(x)}+\frac{1}{\beta_{k}(k^2/x)}\right]\mu^{\prime}(x)=x\frac{\beta^{\prime}_{k}(x)}{[\beta_{k}(x)]^2}
-\frac{k^2}{x}\frac{\beta^{\prime}_{k}(k^2/x)}{[\beta_{k}(k^2/x)]^2}.
\end{align}
In view of Lemma \ref{2.7-lem}, then $\mu^{\prime}(x)<0,~x\in (k,\infty)$ and $\mu^{\prime}(x)>0,~x\in (0,k)$. Hence, $\mu(x)$ is decreasing on $(k,\infty)$ and increasing on $(0,k)$. By using the properties of the logarithmic function, the proof is complete.

\end{proof}

\begin{remark}
Putting $k=1$ in Theorem \ref{5.6-thm}, we get an elegant inequality for classical Nielsen's $\beta$-function
\begin{equation}\label{111-eq}
 \frac{2\beta(x)\beta(1/x)}{\beta(x)+\beta(1/x)}\leq \ln2.
\end{equation}
This inequality \ref{111-eq} was firstly posed by Nantomah in \cite{nan4}.
In 2021, Mate\'{j}\u{c}ka \cite{mate} proved that this inequality held true. Later, Nantomah also offered a new proof by using different method \cite{nan3}.

\end{remark}
Finally, we pose an open problem?
\begin{open}
For $k>0$ and $n\in \mathbb{N}$, show the function $\frac{(x\beta_{k}(x))^{(n+1)}}{(x\beta_{k}(x))^{(n)}(x\beta_{k}(x))^{(n+2)}}$ is decreasing and increasing for $x>0$?
\end{open}
\section*{Methods and Experiment}

Not applicable.

\section*{Competing interests}
The authors declare that they have no competing interests.

\section*{Funding}

This work was supported by the
Shandong Provincial Natural Science Foundation (Grant No. ZR2021MA044) and by the Major Project of Binzhou University (Grant No.2020ZD02)..

\section*{Authors' contributions}
All authors contributed equally to the manuscript and read and approved the final manuscript.

\section*{Acknowledgements}
The authors would like to thank the editor and the anonymous referee for
their valuable suggestions and comments, which help us to improve this paper
greatly.

\end{document}